\documentclass[11pt,  reqno]{amsart}

\usepackage{amsmath,amssymb,amscd,amsthm,amsxtra, esint}

\allowdisplaybreaks[2]

\newtheorem{theorem}{Theorem}% [section]

\newtheorem{corollary}[theorem]{Corollary}
\newcommand{\noi}{\noindent}

\newcommand{\R}{\mathbb{R}}

\newcommand{\dx}{\partial_x}
\newcommand{\dt}{\partial_t}

\newcommand{\jb}[1]
{\langle #1 \rangle}

\begin{document}
					
\date{}

\title
[gKdV equations]
{Nonexistence of soliton-like solutions for defocusing generalized KdV equations}

\author{Soonsik Kwon}
\address{Department of Mathematical Sciences\\
Korea Advanced Institute of Science and Technology\\
291 Daehak-ro Yuseong-gu, Daejeon 305-701, Republic of Korea}
\email{soonsikk@kaist.edu}

\author{Shaunglin Shao}
\address{ 
Department of Mathematics, University of Kansas, Lawrence, KS 66045, USA}
\email{slshao@math.ku.edu}

\begin{abstract}
We consider the global dynamics of the defocusing generalized KdV equations: 
$$ \dt u + \dx^3 u = \dx(|u|^{p-1}u). $$ 
We use Tao's theorem \cite{Tao} that the energy moves faster than mass to prove a moment type dispersion estimate. As an application of the dispersion estimate, we show that there is no soliton-like solutions with decaying assumption. 
\end{abstract}


\maketitle

In this short note, we prove a dispersion estimate of the second moment type for the defocusing generalized KdV equations:
\begin{equation}\label{gkdv}
\dt u + \dx^3 u = \dx(|u|^{p-1}u),\qquad u: \R\times \R \to \R. 
\end{equation}
As an application, we show that there is no soliton-like solution with decaying condition. \\
The gKdV equation \eqref{gkdv} enjoys the mass and energy conservation laws:
\begin{align*}
M(u) &= \int u^2(x) \,dx  \\
E(u) &= \int \frac 12 u_x^2 + \frac{1}{p+1} |u|^{p+1} \,dx
\end{align*}
\noi The local well-posedness of the Cauchy problem on the energy space $H^1(\R)$ is well known \cite{KPV93} and the energy conservation law implies the global existence. \\
In the focusing case, where the sign of the nonlinear term is opposite, there is the soliton solutions $u(t,x) = Q(x-t)$, where $Q$ is the ground state solution: 
$$ Q(x) = \Big(\frac{p+1}{2\cosh^2(\frac{p-1}{2}x)}  \Big)^{1/(p-1)}.   $$
From the Pohozaev identity, one can show that there is no such soliton solution of permanent form in the defocusing case. Furthermore, it is conjectured that the nonlinear global solution scatters to a linear solution. Indeed,  
$$  \lim_{t\to \pm \infty } \| u(t) - e^{-t\dx^3} u_{\pm} \| \to 0. $$
If it were true, as this describes a concrete asymptotic behavior, it implies that there is no spacially localized solutions such as $L^2$- compact solutions.\footnote{there exists a function $x(t)$ such that for any $\epsilon>0$, there exists $R=R(\epsilon) >0$ such that $\int_{|x-x(t)|>R } u^2(t,x)\,dx < \epsilon.$} But toward this direction, there is only a partial result \cite{KKSV}. \\
The purpose of this note is to show an intermediate version. We prove the nonexistence of soliton-like solutions. Main ingredient is the fact that the energy moves faster than the mass to the left. \\
Define the center of mass and the center of energy 
\begin{align*}
\jb{x}_{M}(t) &= \frac 1{M(u)}\int xu^2(t,x) \,dx, \\
\jb{x}_E(t) &= \frac 1{E(u)}\int x (\frac 12 u_x^2 + \frac{1}{p+1} |u|^{p+1}) \,dx.
\end{align*}
Tao \cite{Tao} showed the following monotonicity estimate regarding the center of mass and the center of energy.  

\begin{theorem}[Tao \cite{Tao}]\label{th:tao} Let $p \ge \sqrt{3}$. There exist $c=c(M,E) >0 $ such that
\begin{equation}
\dt\jb{x}_M -\dt\jb{x}_E > c(M,E).
\end{equation}
In particular, we have the dispersion estimate: for any function $x(t)$,
\begin{equation}\label{tao dispersion}
\sup_{t\in \R} \int |x-x(t)|(\rho(t,x) + e(t,x)) \,dx =\infty 
\end{equation}
where $ \rho(t,x) = u^2(t,x)$ and $e(t,x) = \frac 12 u_x^2 + \frac{1}{p+1}|u|^{p+1} $. 
\end{theorem}
This theorem shows that the center of energy moves faster than the center of mass. This behavior is intuitive. From the stationary phase of the linear equation $ u_t + u_{xxx} =0$, one can observe that the group velocity is $ -3\xi^2$, where $\xi$ is the frequency of the wave. Group velocity is negative definite and so every wave moves to the left. Moreover, the higher frequency waves move faster than low frequency waves. Since the energy is more weighted on high frequencies than mass, the center of energy moves faster to the left. The second part of Theorem~\ref{th:tao} is a result from the fact that the distance between $\jb{x}_M $ and $\jb{x}_E $ goes to infinity. \\
We use this property to study a dispersion estimate of moment type.

\begin{theorem}\label{th:dispersion}
Let $p\ge \sqrt{3}$. Let $u(t,x)$ be a nonzero global Schwartz solution to \eqref{gkdv}. Then for any function $x(t)$, 
\begin{equation}\label{eq:dispertion}
\sup_{t\in \R} \int (x-x(t))^2 u^2(t,x) \,dx = \infty. 
\end{equation} 

\end{theorem} 

This can be seen as an improvement of \eqref{tao dispersion}, since we use solely the mass density. Roughly speaking, Theorem~\ref{th:dispersion} tells that the mass cannot be localized around the center of mass, but has to spread out in time. Usually, such a dispersion behavior is characterized as time decay of solutions or the boundedness of space-time norms, such as the Strichartz estimates. This provides another form of dispersion estimate.\\
As a corollary, we observe that there is no soliton-like solution under decaying assumption. 
\begin{corollary}\label{co:nonexistence}
Assume that $u(t,x)$ is a global soliton-like solution in the sense that there exists $x(t) \in \R $ such that for any $R>0$,  
$$  \sup_{t\in \R} \int_{|x-x(t)|>R} u^2(t,x)\,dx \lesssim \frac{1}{R^{2+\epsilon}}.
$$
Then $ u\equiv 0$.
\end{corollary}

There are some works of this type. de Bouard and Martel \cite{deBouard-Martel} showed for the KP-II equation  the nonexistence of $L^2$- compact solutions under certain positivity condition on $x'(t)$. Their work can be written for the defocusing gKdV equation with $x'(t) > 0$ condition. This can read that there is no soliton-like solution moving to the right, as a real soliton solution moves to the right. Here, we do not specify a direction. In \cite{Martel-Merle}, Martel and Merle assume a similar decaying condition, and show the nonexistence of minimal mass blow-up solutions for critical gKdV equation (p=5).  \\ 
In the rest of the note, we provide the proof of Theorem~\ref{th:dispersion} and Corollary~\ref{co:nonexistence}.

\begin{proof}[Proof of Theorem~\ref{th:dispersion}]  $\,$ \\
As $\jb{x}_M = \frac{1}{M(u)} \int xu^2(x)\,dx$ is a critical point of 
$$ f(a) = \int (x-a)^2u^2(x)\,dx, $$ $\int (x-x(t))^2 u(t,x) \,dx $ is minimized at $x(t) = \jb{x}_M $. So, it suffices to show 
\begin{equation}\label{dispersion}
\sup_{t\in \R} \int (x-\jb{x}_M)^2 u^2(t,x) \,dx = \infty. 
\end{equation}
We use the equation \eqref{gkdv} and integration by parts to compute
\begin{align*}
\frac{d}{dt}\int &(x-\jb{x}_M)^2 u^2(t,x)\,dx \\
&= -\int 2(x-\jb{x}_M)  u^2 \,dx\cdot \frac{d}{dt}\jb{x}_M + \int (x-\jb{x}_M)^2 2uu_t \, dx \\
       &= 0 + \int (x-\jb{x}_M)^2 2u(-u_{xxx} + \dx(|u|^{p-1}u) )\,dx \\
       &\ge -6 \int u_x^2(x-\jb{x}_M) \,dx -4 \int (x-\jb{x}_M)|u|^{p+1}\,dx + \frac{4}{p+1} \int (x-\jb{x}_M) |u|^{p+1} \,dx \\
       &=-12\int \Big(\frac 12 u_x^2 + \frac{1}{p+1}|u|^{p+1}\Big)(x-\jb{x}_M) \,dx - \frac{4p-12}{p+1} \int |u|^{p+1}(x-\jb{x}_M) \,dx \\
       &= -12E(u) \Big( \jb{x}_E - \jb{x}_M \Big) - \frac{4p-12}{p+1} \int|u|^{p+1}(x-\jb{x}_M) \,dx 
\end{align*}

We show \eqref{dispersion} by contradiction, assuming that
$$ \sup_{t\in \R} \int (x-\jb{x}_M)^2 u^2(t,x) \,dx < C. $$
The second term is bounded due to Sobolev embedding and conversation laws: 

\begin{align*}
 \int|u|^{p+1}(x-\jb{x}_M)\,dx &\le \|u\|^{p-1}_{L^\infty}\Big( \int u^2(t,x)(x-\jb{x}_M)^2 \,dx +M(u) \Big) \\
                           &\le  2(E(u)+M(u))(C+ M(u)) \le C_1.
\end{align*}

From Theorem~\ref{th:tao}, since $\jb{x}_E-\jb{x}_M$ monotonically decreases, we have eventually 
$$ \frac{d}{dt}\int (x-\jb{x}_M)^2 u^2(t,x)\,dx \ge -12E(u)(\jb{x}_E-\jb{x}_M) -C_1 > 0. $$ This makes a contradiction.  
\end{proof} 

\begin{proof}[Proof of Corollary~\ref{co:nonexistence}] $ $\\
We simply estimate
\begin{align*}
\int (x-x(t))^2 u^2(t,x)\,dx &\le M(u) + \sum_{k=0}^\infty \int_{\{2^{k+1}> |x-x(t)|\ge 2^{k}\}}(x-x(t))^2 u^2(t,x)\,dx \\
                         & \lesssim M(u) + \sum_{k=0}^\infty  2^{2(k+1)}\cdot 2^{(-2-\epsilon)k} < \infty. 
\end{align*}
Hence, by Theorem~\ref{th:dispersion}, $u\equiv 0$.

\end{proof}

%\begin{remark}
%This moment type dispersion estimate is weaker than space-time norm bound.
%In view of 

%\end{remark}

\noi \textbf{Acknowledgements} \\ 
S.K. is partially supported by NRF(Korea) grant 2010-0024017. S.S. is partially supported by DMS-1160981.

%\section{}
%\subsection{}

\end{document}